\documentclass[12pt]{amsart}

\usepackage{graphicx}
\usepackage{subfigure}
\usepackage{eepic}
\usepackage{xcolor}
\selectcolormodel{gray}
\usepackage{tikz}
\usepackage{amsmath}
\usepackage{a4wide}
\usepackage[utf8]{inputenc}
\usepackage{amssymb}
\usepackage{amsopn}
\usepackage{epsfig}
\usepackage{amsfonts}
\usepackage{latexsym}
\usepackage{amsthm}
\usepackage{enumerate}
\usepackage[UKenglish]{babel}
\usepackage{verbatim}
\usepackage{color}

\newtheorem{theorem}{Theorem}[section]
\newtheorem{lemma}[theorem]{Lemma}
\newtheorem{proposition}[theorem]{Proposition}

\theoremstyle{definition}

\theoremstyle{remark}
\newtheorem{remark}[theorem]{Remark}

\numberwithin{equation}{section}

\renewcommand{\emptyset}{\varnothing}

\begin{document}

\title[Positive Lebesgue measure and empty interior]{An infinitely generated self-similar set with positive Lebesgue measure and empty interior}

\author{Simon Baker and Nikita Sidorov}
\address{Mathematics institute, University of Warwick, Coventry, CV4 7AL, UK}
\email{simonbaker412@gmail.com}

\address{
School of Mathematics, The University of Manchester,
Oxford Road, Manchester M13 9PL, United Kingdom.}
\email{sidorov@manchester.ac.uk}

\date{\today}

\subjclass[2010]{}

\begin{abstract}
In \cite{PS} Peres and Solomyak asked the question: Do there exist self-similar sets with positive Lebesgue measure and empty interior? This question was answered in the affirmative by Cs\"{o}rnyei et al in \cite{CJPPS}. The authors of that paper gave a parameterised family of iterated function systems for which almost all of the corresponding self-similar sets satisfied the required properties. They did not however provide an explicit example. Motivated by a desire to construct an explicit example, we provide an explicit construction of an infinitely generated self-similar set with positive Lebesgue measure and empty interior.
\end{abstract}

\keywords{Self-similar sets, Lebesgue measure, Interior}
\maketitle

\section{Introduction}
We call a map $\phi:\mathbb{R}^d\to\mathbb{R}^d$ a contracting similarity if $|\phi(x)-\phi(y)|=r|x-y|$ for all $x,y\in\mathbb{R}^d$ for some $r\in(0,1)$. A well known result due to Hutchinson \cite{Hut} states that given a finite collection of contracting similarities $\Phi:=\{\phi_i\}_{i=1}^l$, there exists a unique, non-empty, compact set $K\subset \mathbb{R}^d$ satisfying $$K=\bigcup_{i=1}^{l}\phi_{i}(K).$$ Such a $K$ is called the self-similar set generated by $\Phi$. In what follows we refer to a finite collection of contracting similarities as an iterated function system or IFS for short. When there exists an open set $O\subset \mathbb{R}^d$ satisfying $\phi_i(O)\subset O$ for all $1\leq i\leq l$, and $\phi_{i}(O)\cap \phi_j(O)=\emptyset$ for all $i\neq j,$ much can be said about the metric properties of $K$. This assumption is called the open set condition. Assuming the open set condition, if we let $r_i$ denote the contraction rate of $\phi_i,$ then the Hausdorff dimension of $K$ is the unique $s\geq 0$ satisfying $\sum_{i=1}^lr_i^s=1$. In \cite{Sch} Schief proved that if the open set condition holds and $\sum_{i=1}^lr_i^d=1,$ then $K$ not only has positive Lebesgue measure but also has non-empty interior. This naturally gave rise to the following question posed in \cite{PS} by Peres and Solomyak: Do there exist self-similar sets with positive Lebesgue measure and empty interior? This question was answered in the affirmative by Cs\"{o}rnyei et al in \cite{CJPPS}. Their approach relied on ideas due to Jordan and Pollicott \cite{JP}. In \cite{CJPPS} the authors constructed a parameterised family of IFSs for which almost all (in the sense of some measure) of the corresponding self-similar sets have positive Lebesgue measure and empty interior. Importantly they did not provide an explicit example. The motivation behind this paper comes from a desire to construct an explicit example of a self-similar set satisfying the desired properties. We do not succeed in this goal, but instead construct an example of an infinitely generated self-similar set with positive Lebesgue measure and empty interior. Importantly, as we will explain below, a suitable analogue of the question of Peres and Solomyak persists for infinitely generated self-similar sets. In which case our main result does provide an explicit example for this question.


Given a countable collection of contracting similarities $\Phi:=\{\phi_i\}_{i=1}^{\infty},$ we say that $\Phi$ is contractive if $\sup_i\{r_i\}<1$. Moran showed in \cite{Mor} that when $\Phi$ is contractive and relatively compact in the space of contractions on $\mathbb{R}^d$ endowed with the topology of uniform convergence over bounded sets, then one can define an appropriate analogue of a self-similar set for $\Phi$. We call this analogue an infinitely generated self-similar set. Note that if $K$ is an infinitely generated self-similar set for a relatively compact, contractive $\Phi$, then $K$ satisfies the relation $K=\cup_{i=1}^{\infty}\phi_{i}(K)$. When $\Phi$ is compact then the analogy with self-similar sets can be made more precise via a result of Wicks from \cite{Wicks}. In \cite{Wicks} it is shown that if $\Phi$ is contractive and compact then the set valued map sending $X$ to $\bigcup_{i=1}^{\infty}\phi_i(X)$ is a contraction on the space of non-empty compact sets equipped with the Hausdorff metric. Moreover, there exists a unique non-empty compact set $K\subset \mathbb{R}^d$ satisfying $$K=\bigcup_{i=1}^{\infty}\phi_i(K).$$ The set $K$ is the infinitely generated self-similar set of $\Phi$ as defined by Moran. 

Moran proved in \cite{Mor} that if $\Phi$ is contractive, relatively compact, satisfies the open set condition, and $\sum_{i=1}^{\infty}r_i^d=1,$ then the infinitely generated self-similar set has positive Lebesgue measure. When $\Phi$ is also compact, it can be shown that the proof of Schief from \cite{Sch} still applies, and one can show that the infinitely generated self-similar set must have non-empty interior. As such the motivating question of Peres and Solomyak has a natural analogue for infinite iterated function systems: Do there exist infinitely generated self-similar sets with positive Lebesgue measure and empty interior? The main result of this paper answers this question in the affirmative and provides an explicit example.


In this paper we consider the following collection of similarities:

\begin{align*}
U(x,y)&:=\left(\frac{x}{2},\frac{y+1}{2}\right),\\
D_{0}(x,y)&:=\left(\frac{x}{2},\frac{y}{2}\right),\\
D_{k,n}(x,y)&:=\left(\frac{x+t_{k,n}}{2},\frac{y}{2}\right),\, \textrm{ where }t_{k,n}:=\frac{1}{2^k\cdot n},\,k\geq 0,n\geq 1.
\end{align*} We let $\Phi$ denote this set of contractions. $\Phi$ is compact in the space of contractions on $\mathbb{R}^2.$ This is because a sequence of contractions in $\Phi$ either contains $D_{k,n}$ for arbitrarily large values of $k$ or $n,$ or the sequence contains only finitely many elements from $\Phi.$ In the latter case the sequence trivially has a convergent subsequence. In the former case it is clear that one can pick a subsequence converging to $D_0$. Therefore $\Phi$ is compact and we can apply the result of Wicks to assert that there exists a unique non-empty compact $K\subset\mathbb{R}^2$ satisfying
\begin{equation}
\label{decomposition}
K=U(K)\cup D_0(K)\cup \bigcup_{\substack{k\geq 0\\ n\geq 1}} D_{k,n}(K).
\end{equation}
In Figure~\ref{Fig1} below we include an approximation of the set $K$ using finitely many of the maps in $\Phi$.

\begin{figure}
\includegraphics[width=350pt]{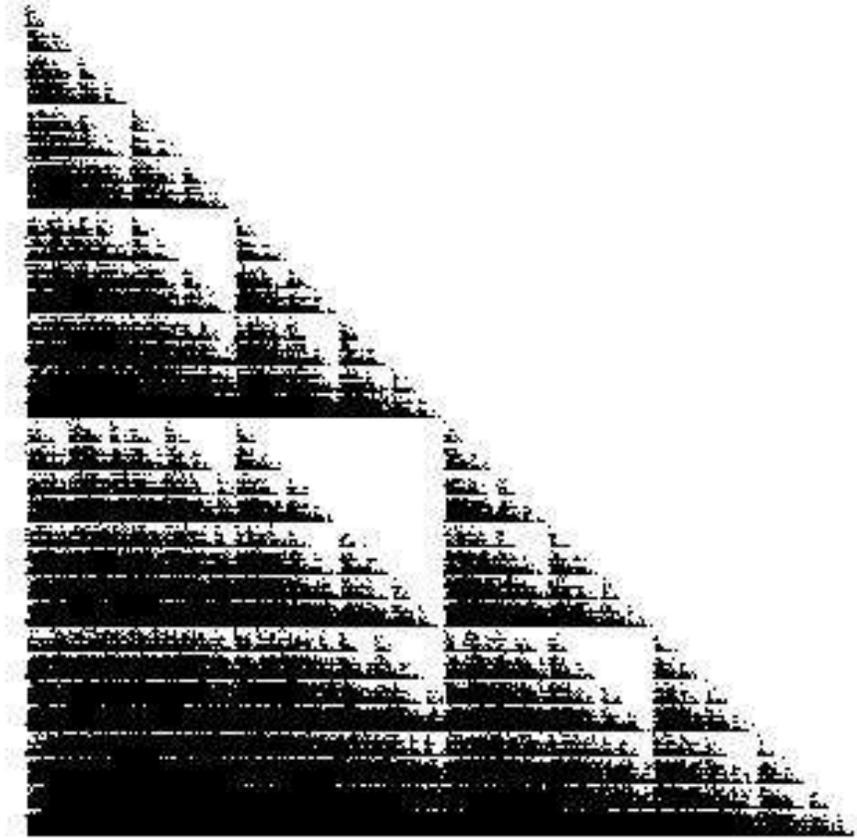}
\caption{An approximation of $K$ using finitely many similarities.}
\label{Fig1}
\end{figure}


 Given a sequence $(a_i)\in\{0,1\}^{\mathbb{N}}$ let $Z(a_i):=\{i\in\mathbb{N}:a_i=0\}$. Using this notation, we have the following simple expression for $K$: 
 \begin{equation}
 \label{simple}
 K=\left\{\left(\sum_{i\in Z(a_i)}\frac{d_i}{2^{i}},\sum_{i=1}^{\infty}\frac{a_i}{2^i}\right): (a_i)\in\{0,1\}^{\mathbb{N}},\, d_i\in \{0\}\cup\{t_{k,n}\}_{\substack{k\geq 0\\ n\geq 1}}\right\}.
 \end{equation} One can verify that this equation is correct by observing that the right-hand side satisfies \eqref{decomposition}. In this paper we prove the following theorem.
\begin{theorem}
	\label{Main theorem}
$K$ has positive Lebesgue measure and empty interior. Moreover, Lebesgue almost every nontrivial horizontal fibre contains an interval.
\end{theorem}
\section{Proof of Theorem~\ref{Main theorem}}
We split our proof of Theorem~\ref{Main theorem} into two parts. We begin by proving that $K$ has empty interior.

\begin{lemma}
	\label{compact}
	For any $m\in\mathbb{N}$ the set 
\begin{equation}\label{eq:Xm}
X_m:=\left\{\sum_{i=1}^m\frac{d_i}{2^{i}}: d_{i}\in\{0\}\cup \{t_{k,n}\}_{\substack{k\geq 0\\ n\geq 1}}\right\}
\end{equation}
is compact.
\end{lemma}
\begin{proof}
Let $(x_l)_{l=1}^{\infty}$ be a sequence in $X_m$. Each $x_l$ can be written as 
$$
x_l:=\sum_{i=1}^m\frac{d_{i}^{(l)}}{2^i},
$$ 
for some $\bigl(d_i^{(l)}\bigr)\in\left\{\{0\}\cup \{t_{k,n}\}_{\substack{k\geq 0\\ n\geq 1}}\right\}^m.$ If $(d_1^{(l)})$ consists of infinitely many different terms then there must exist a subsequence $(d_1^{(l_j)})$ such that $d_1^{(l_j)}\to 0$ as $j\to \infty$. Alternatively, $(d_1^{(l)})$ consists of finitely many terms and there exists $c_1\in \{0\}\cup \{t_{k,n}\}_{\substack{k\geq 0\\ n\geq 1}},$ and a subsequence $\bigl(d_1^{(l_j)}\bigr)$ such that  $d_{1}^{(l_j)}=c_1$ for all $j$. In either case we have found a $c_1\in \{0\}\cup \{t_{k,n}\}_{\substack{k\geq 0\\ n\geq 1}}$ and a subsequence $(d_1^{(l_j)})$ such that $d_1^{(l_j)}\to c_1$ as $j\to\infty$. Replacing $(d_1^{(l)})$ with $(d_2^{(l_j)})$ in the above argument, we see that there exists a subsequence of $(d_2^{(l_j)})$ which converges to an element of $\{0\}\cup \{t_{k,n}\}_{\substack{k\geq 0\\ n\geq 1}}$. Repeating this argument for each remaining component, we may assert that there exists $(c_1,\ldots,c_m)\in \left\{\{0\}\cup \{t_{k,n}\}_{\substack{k\geq 0\\ n\geq 1}}\right\}^m$ and a subsequence $(l_j^{*})$ such that $$\sum_{i=1}^m\frac{d_{i}^{(l_j^*)}}{2^i}\to \sum_{i=1}^m\frac{c_i}{2^i}.$$ Therefore $X_m$ is compact.
	
\end{proof}
\begin{lemma}
	\label{nowhere dense}
	For any $m\in\mathbb{N}$, the set $X_m$ given by (\ref{eq:Xm}) 
is nowhere dense in $\mathbb{R}$.
\end{lemma}
\begin{proof}
If $X_m$ were dense within some subinterval of $\mathbb{R},$ it would follow from the compactness of $X_m$ established in Lemma~\ref{compact} that $X_m$ would in fact contain this subinterval. However $X_m$ is countable, and so this is not possible.
\end{proof}Recall that $\Phi$ is the set of similarities defining $K$. In what follows it will be useful to enumerate this set in the following way, $\Phi:=\{\phi_i\}_{i=1}^{\infty}$.

Let $\Delta$ be the right angled triangle uniquely determined by the three points $(0,0)$, $(1,0)$, and $(0,1)$. Before proceeding with our proof that $K$ has empty interior we state some properties of $\Delta$. It is a straightforward observation that for any $m\in\mathbb{N}$ we have 
\[
\bigcup_{(\phi_i)_{i=1}^m\in \Phi^m}(\phi_1\circ \cdots \circ \phi_m)(\Delta) \subseteq \bigcup_{(\phi_i)_{i=1}^{m-1}\in \Phi^{m-1}}(\phi_1\circ \cdots \circ \phi_{m-1})(\Delta) \subseteq \cdots \subseteq \bigcup_{\phi_i\in \Phi}\phi_i(\Delta)\subseteq \Delta.
\] 
Therefore, since the set valued map $X\mapsto \bigcup_{\phi_i\in\Phi} \phi_i(X)$ is a contraction on the space of non-empty compact subsets equipped with the Hausdorff topology, we may assert that for any $m\in\mathbb{N}$ we have
\begin{equation}
\label{inclusion}
K\subseteq \bigcup_{(\phi_i)_{i=1}^m\in \Phi^m}(\phi_1\circ \cdots \circ \phi_m)(\Delta).
\end{equation} We also remark that
\begin{equation}
\label{strip}(\phi_1\circ \cdots \circ \phi_m)(\Delta)\subseteq \mathbb{R}\times \left[\,\sum_{i:\phi_i= U}\frac{1}{2^i},\sum_{i:\phi_i= U}\frac{1}{2^i}+\frac{1}{2^m}\right].
\end{equation}In other words $(\phi_1\circ \cdots \circ \phi_m)(\Delta)$ is contained in a unique horizontal strip of height $2^{-m}$ which is determined by the occurrences of the similarity $U$ within $(\phi_1,\ldots,\phi_m)$.
\begin{proposition}
	\label{interior}
	The set $K$ has empty interior.
\end{proposition}
\begin{proof}
For the purpose of deriving a contradiction suppose that $K$ has nonempty interior. Then there exists $I\times J\subseteq K$ for two nontrivial open intervals $I$ and $J$. Without loss of generality we may assume that $$J=\left(\sum_{i=1}^m\frac{a_i}{2^i},\sum_{i=1}^m\frac{a_i}{2^i}+\frac{1}{2^m}\right)$$ for some $m\in\mathbb{N}$ and $(a_i)\in\{0,1\}^m.$ Let $\Phi_J:=\{(\phi_i)_{i=1}^m: \phi_i=U \textrm{ if and only if } a_i=1\}.$ It follows from \eqref{inclusion} and \eqref{strip} that
	\begin{equation}
	\label{2bcontradicted}I\times J\subseteq \bigcup_{(\phi_i)_{i=1}^m\in \Phi_J} (\phi_1\circ \cdots\circ\phi_m)(\Delta).
	\end{equation}
	
	Since $X_m$ is nowhere dense by Lemma~\ref{nowhere dense}, it follows that there exists an $x\in I$ and $r\in(0,2^{-m})$ such that $(x-r,x+r)\subset I$ and $(x-r,x+r)\cap X_m=\emptyset$. The $x$-coordinate of  $(\phi_1\circ \cdots\circ\phi_m)((0,1))$ is the image of $0$ under the concatenation of $m$ maps from the set $\Bigl\{\frac{x+d_i}{2}:d_i\in\{0\}\cup\{t_{k,n}\}_{\substack{k\geq 0\\ n\geq 1}}\Bigr\}$. It follows that the $x$-coordinate of $(\phi_1\circ \cdots\circ\phi_m)((0,1))$ is contained in $X_m$.  Therefore
	\begin{equation}
	\label{empty}
	\left[(x-r,x+r)\times \left\{\sum_{i=1}^{m}\frac{a_i}{2^i}+\frac{1}{2^m}\right\}\right]\cap 	\bigcup_{(\phi_i)_{i=1}^m\in \Phi_J} (\phi_1\circ \cdots\circ\phi_m)((0,1))=\emptyset.
	\end{equation} Since each $(\phi_1\circ \cdots\circ\phi_m)(\Delta)$ is a right angled triangle whose two other angles are $\pi/4$ and whose apex is at $(\phi_1\circ \cdots\circ\phi_m)((0,1))$ with $y$-coordinate $\sum_{i=1}^{m}\frac{a_i}{2^i}+\frac{1}{2^m},$  it follows from \eqref{empty} and simple geometric considerations that $$\Big(x-\frac{r}{3},x+\frac{r}{3}\Big)\times \Big(\sum_{i=1}^m\frac{a_i}{2^i}+\frac{1}{2^m}-\frac r3,\sum_{i=1}^m\frac{a_i}{2^i}+\frac{1}{2^m}\Big)$$ has empty intersection with $$ \bigcup_{(\phi_i)_{i=1}^m\in \Phi_J} (\phi_1\circ \cdots\circ\phi_m)(\Delta).$$ This contradicts \eqref{2bcontradicted}. Therefore we may conclude that $K$ has empty interior.		
\end{proof}
\begin{remark}
	Proposition \ref{interior} holds more generally when $\Phi$ is any IFS consisting of $U$ and $D_j(x,y)=\left(\frac{x+s_j}{2},\frac{y}{2}\right)$ where $\{s_j\}_{j=1}^{\infty}$ is a countable compact subset of $[0,1]$ containing zero.
\end{remark}
We now move on to the second half of the proof of Theorem~\ref{Main theorem}, namely proving that $K$ has positive Lebesgue measure and Lebesgue almost every nontrivial horizontal fibre of $K$ contains an interval. We start with the following straightforward lemma.

\begin{lemma}
	\label{base8}
	For any $x\in[0,1/56]$ there exists $(\tilde{d}_i(x))\in\{\{0\}\cup \{1/n\}_{n\geq 1}\}^{\mathbb{N}}$ such that $$x=\sum_{i=1}^{\infty}\frac{\tilde{d}_i(x)}{8^i}.$$
\end{lemma}
\begin{proof}
For any $x\in [0,\frac{1}{56}]$ we have $8x\in [0,\frac{1}{7}].$ For any $n\geq 7$ one can check that $\frac{1}{n}-\frac{1}{n+1}=\frac{1}{n(n+1)}\leq \frac{1}{56}.$ Therefore, picking the largest element of $\{0\}\cup \{1/n\}_{n\geq 1}$ less than $8x$ yields a $\tilde{d}_1(x)\in \{0\}\cup \{1/n\}_{n\geq 1}$ such that $8x-\tilde{d}_1(x)\in [0,\frac{1}{56}].$ Let $x_1=8x-\tilde{d}_1(x)$, then $$x=\frac{\tilde{d}_1(x)}{8}+\frac{x_1}{8}.$$ Since $x_1\in[0,\frac{1}{56}]$ there exists $\tilde{d}_2(x)\in \{0\}\cup \{1/n\}_{n\geq 1}$ such that $8x_1-\tilde{d}_2(x)\in[0,\frac{1}{56}]$. Letting $x_2=8x_1-\tilde{d}_2(x)$ and substituting this equation into the above yields  $$x=\frac{\tilde{d}_1(x)}{8}+\frac{\tilde{d}_2(x)}{64}+\frac{x_2}{64}.$$ We then repeat the above step with $x_1$ replaced with $x_2.$ Clearly we can repeat this step indefinitely. Doing so yields a sequence $(\tilde{d}_i(x))$ with the desired properties.
\end{proof} 

\begin{remark}It is easy to see that one can take $(\tilde{d}_i(x))\in\{\{0\}\cup \{1/n\}_{n=1}^{248}\}^{\mathbb{N}}$, with the same conclusion. We will keep our set of $n$ infinite though, because
the set of $k$ in Theorem~\ref{Main theorem} cannot be made finite using our argument (see below). 
\end{remark}

Examining the representation obtained in Lemma~\ref{base8}, we see that it can be reinterpreted as a representation in base $2$ where the only nonzero digits are those occurring at positions within $3\mathbb{N}$. The fact that we can represent any $x\in[0,1/56]$ in base $2$ using a sequence where non-zero digits occur at most $1/3$ of the time will be a useful tool when it comes to showing that generic fibres contain intervals.

For each $y\in[0,1]$ let $$K_y:=\{x:(x,y)\in K\}.$$

\begin{proposition}
	\label{fibre}
For Lebesgue almost every $y\in[0,1]$ the set $K_y$ contains an interval.
\end{proposition}
\begin{proof}
Each $y\in[0,1]$ has a binary expansion $(a_i)\in\{0,1\}^{\mathbb{N}}$ satisfying $y=\sum_{i=1}^{\infty}a_i2^{-i}$. This expansion is unique for Lebesgue almost every $y\in[0,1]$. Moreover, it follows from the strong law of large numbers that the binary expansion of Lebesgue almost every $y\in[0,1]$ satisfies:
\begin{equation}
\label{normal}
\lim_{N\to\infty}\frac{\#\{1\leq i\leq N:a_i=0\}}{N}=\frac{1}{2}.
\end{equation}
Let $$A_N:=\Big\{y\in[0,1]:\#\{1\leq i\leq n:a_i=0\}\geq 0.4\cdot n\, ,\forall n\geq N \Big\}.$$ It follows from \eqref{normal} that for any $\epsilon>0,$ we can pick $N$ sufficiently large that $\mathcal{L}(A_{N})>1-\epsilon$. Here and throughout $\mathcal{L}$ denotes the one-dimensional Lebesgue measure. In what follows we assume that $\epsilon>0$ has been fixed and we have chosen $N$ sufficiently large.

Fix $y\in A_{N}$ and let $(a_i)\in\{0,1\}^{\mathbb{N}}$ be its binary expansion, which we may assume is unique. It follows from \eqref{simple} that $K_y$ has the following form:
\begin{equation}
\label{fibreequation}
K_y=\Big\{\sum_{i\in Z(a_i)}\frac{d_i}{2^{i}}: d_i\in \{0\}\cup\{t_{k,n}\}_{\substack{k\geq 0\\ n\geq 1}}\Big\},
\end{equation} where $Z(a_i)=\{i\in\mathbb{N}:a_i=0\}.$

Consider the interval $I_N:=[0,\frac{1}{56\cdot 2^N}].$ We will show that $I_N\subseteq K_{y}.$ It follows from Lemma~\ref{base8} that for each $x\in I_N,$ there exists a sequence $(d_i^*(x))\in\{\{0\}\cup \{1/n\}_{n\geq 1}\}^{\mathbb{N}}$ satisfying:
\begin{align}
x&=\sum_{i=1}^{\infty}\frac{d_i^*(x)}{2^i}, \label{equality}\\
d_i^*(x)&=0\, \textrm{ for } 1\leq i \leq N,\label{Nzeros}\\
d_i^*(x)&\neq 0\implies i-N=0\, (\textrm{mod }3)\, \label{Mod3}.
\end{align}

Fix $x\in I_N$ and let $(d_i^*(x))$ satisfy the above. In what follows we will assume that $d_i^*(x)\neq 0$ for infinitely many $i$. The case where $(d_i^*(x))$ ends with infinitely many zeros is handled similarly. To prove that $x\in K_y$ we will construct a bijection $f:Z(a_i)\to \{i:d_i^*(x)\neq 0\}$ such that $f(i)>i$ for all $i\in Z(a_i).$

To construct $f$, we first define the following sets. For each $k\geq 1$ let $$W_k:=\{kN+1\leq i \leq (k+1)N:d_i^*(x)\neq 0\}.$$ Let $S_{1}$ be the first $\#W_1$ elements of the set $\{1\leq i \leq N:a_i=0\},$ and for $k\geq 2$ let $S_k$ be the first $\#W_{k}$ elements of the set 
$$
\left\{1\leq i \leq kN:a_i=0\, \textrm{ and }i\notin \bigcup_{l=1}^{k-1}S_{l}\right\}.
$$ 
We claim that for each $k\geq 1$ the set $S_k$ is well defined and there is a bijection $f_k:S_k\to W_k$ satisfying $f_k(i)>i$ for all $i\in S_k$. We prove this claim via induction.
\\

\noindent \textbf{Step $1$.} It follows from \eqref{Mod3} that $\#W_1 \leq \lceil \frac{N}{3} \rceil.$ Here and throughout we let  $\lceil x \rceil:=\inf\{n\in\mathbb{N}:x\leq n\}$.  It follows from the definition of $A_N$ that $\#\{1\leq i\leq N:a_i=0\}\geq 0.4N$. Without loss of generality we may assume that $N$ is sufficiently large that $0.4N>\lceil \frac{N}{3}\rceil\geq \#W_1 $. For such an $N$ the set $S_1$ is well defined. We let $f_1:S_1\to W_1$ be an arbitrary bijection. Since every $i\in S_1$ satisfies $i\leq N$, and every element of $W_1$ is greater than or equal to $N+1,$ we may conclude that $f_1(i)>i$ for all $i\in S_1$.
\\

\noindent \textbf{Step $k+1$.}  Assume that for each $1\leq l \leq k$ the set $S_l$ is well defined. Moreover, assume that for each $1\leq l\leq k$ we have constructed a bijection $f_l:S_l\to W_l$ such that $f_l(i)>i$ for all $i\in S_l$. We now make our inductive step.

It follows from the definition of $A_{N}$ and $S_l$ that
\begin{align}
\label{bound}
\#\big\{1\leq i \leq (k+1)N:a_i=0\, \textrm{ and }i\notin \cup_{l=1}^{k}S_{l}\big\}&\geq 0.4(k+1)N-\sum_{l=1}^k\#W_j\nonumber\\ &\geq 0.4(k+1)N-k\Big\lceil \frac{N}{3}\Big\rceil \nonumber\\ &>\Big\lceil \frac{N}{3}\Big\rceil.
\end{align}By \eqref{Mod3} we have $\#W_{k+1}\leq \lceil \frac{N}{3}\rceil.$ Combining this observation with \eqref{bound} we see that the set $S_{k+1}$ is well defined. Let $f_{k+1}:S_{k+1}\to W_{k+1}$ be an arbitrary bijection. Since every element of $S_{k+1}$ is less than or equal to $(k+1)N,$ and every element of $W_{k+1}$ is greater than or equal to $(k+1)N+1,$ we have $f_{k+1}(i)>i$ for all $i\in S_{k+1}.$ This completes our inductive step and our claim holds.
\\

To construct our bijection $f:Z(a_i)\to \{i:d_i^*(x)\neq 0\}$ we start by remarking that $\{i:d_i^*(x)\neq 0\}=\bigcup_{k=1}^{\infty}W_k$ by \eqref{Nzeros}, and by our assumption $d_i^*(x)\neq 0$ for infinitely many $i$ we have $Z(a_i)=\bigcup_{k=1}^{\infty}S_k.$ We define our bijection $f:Z(a_i)\to \{i:d_i^*(x)\neq 0\}$ via the rule $f(i)=f_{k}(i)$ if $i\in S_k$. This is a bijection by our construction and the above remarks. Moreover, $f(i)>i$ for all $i\in Z(a_i),$ since $f_k(i)>i$ for all $i\in S_k$ for any $k\geq 1.$

Equipped with our bijection $f$ we can now show that $x\in K_y$. For each $i\in Z(a_i)$ we define
\begin{equation}
\label{substitution}
d_i^{**}:=\frac{d_{f(i)}^*(x)}{2^{f(i)-i}}.
\end{equation} Since $f(i)>i$ for all $i\in Z(a_i)$ and $(d_i^*(x))\in\{\{0\}\cup \{1/n\}_{n\geq 1}\}^{\mathbb{N}},$ we see that $d_i^{**}\in \{0\}\cup \{t_{k,n}\}_{\substack{k\geq 0\\ n\geq 1}}.$ Moreover, $$\sum_{i\in Z(a_i)}\frac{d_i^{**}}{2^{i}}\stackrel{\eqref{substitution}}{=}\sum_{i\in Z(a_i)}\frac{d_{f(i)}^*(x)}{2^{f(i)}}=\sum_{i:d_i^{*}(x)\neq 0}\frac{d_{i}^*(x)}{2^{i}}\stackrel{\eqref{equality}}{=}x.$$ In the second equality we used the fact $f:Z(a_i)\to\{i:d_i^{*}(x)\neq 0\} $ is a bijection. By \eqref{fibreequation} we have $x\in K_y.$ Since $x$ was arbitrary we may conclude that for every $y\in A_N$ the set $K_y$ contains $I_N$. Since $\mathcal{L}(A_N)>1-\epsilon$ and $\epsilon$ is arbitrary, we may conclude that for Lebesgue almost every $y\in[0,1]$ the set $K_y$ contains an interval.
\end{proof}
Theorem~\ref{Main theorem} now follows from Proposition~\ref{interior}, Proposition~\ref{fibre}, and Fubini's theorem.

We remark that the same method used in our proof of Theorem~\ref{Main theorem} can be used to give an explicit example of an infinitely generated self-affine set with positive Lebesgue measure and empty interior, which is not an infinitely generated self-similar set. Indeed, the contractions
\begin{align*}
U(x,y)&:=\left(\frac{x}{\sqrt{2}},\frac{y+1}{2}\right),\\
D_{0}(x,y)&:=\left(\frac{x}{\sqrt{2}},\frac{y}{2}\right),\\
D_{k,n}(x,y)&:=\left(\frac{x+t_{k,n}}{\sqrt{2}},\frac{y}{2}\right),\, \textrm{ where }t_{k,n}:=\frac{1}{2^{k/2}\cdot n},\,k\geq 0,n\geq 1,
\end{align*}yield an infinitely generated self-affine set with the desired properties.
\\

\noindent \textbf{Acknowledgements.} The authors are grateful to the referees whose comments greatly improved the content and exposition of this paper.

\end{document}